\documentclass[11pt]{article}
\usepackage[pagebackref,colorlinks=true,pdfpagemode=none,urlcolor=blue,linkcolor=blue,citecolor=blue]{hyperref}

\usepackage{amsthm, amssymb, bm}
\usepackage{soul, color}
\usepackage[]{amsmath}
\usepackage[]{amsfonts}
\usepackage[]{fancyhdr}
\usepackage[]{graphicx}
\usepackage[]{empheq}

\usepackage[toc,page]{appendix}
\newtheorem{theorem}{Theorem}[section]

\newtheorem{proposition}[theorem]{Proposition}
\newtheorem{corollary}[theorem]{Corollary}

\newtheorem{remark}[theorem]{Remark}
\newtheorem{hypothesis}[theorem]{Hypothesis}

\def \Cm {\mathbb{C}}
\def \Imm {\mathbb{I}}

\def \Rm {\mathbb{R}}

\def\B{\mathcal{B}}
\def\C{\mathcal{C}}

\def\F{\mathcal{F}}
\def\H{\mathcal{H}}

\def\L{\mathcal{L}}

\def\N{\mathcal{N}}
\def\O{\mathcal{O}}

\def\V{\mathcal{V}}
\def\W{\mathcal{W}}

\newcommand{\bk}{\mathbf k}

\newcommand{\where}{\quad\text{ where }}
\newcommand{\qandq}{\quad\text{ and }\quad}

\newcommand{\bfe}{ {\bf e}}
\newcommand{\bff}{ {\bf f}}

\newcommand{\bfh}{ {\bf h}}

\newcommand{\bfu}{ {\bf u}}
\newcommand{\bfv}{ {\bf v}}

\newcommand{\bfI}{ {\bf I}}
\newcommand{\bfP}{ {\bf P}}

\newcommand{\tr}{ {\text{tr }}}

\newcommand{\cout}[1]{}

\def \bmrho {{\boldsymbol\rho}}
\def \bmphi {{\boldsymbol\varphi}}

\title{Inverse anisotropic conductivity from internal current densities}
\author{Guillaume Bal\thanks{Department of Applied Physics and Applied Mathematics, Columbia University,  New York NY, 10027; gb2030@columbia.edu} \and Chenxi Guo\thanks{Department of Applied Physics and Applied Mathematics, Columbia University,  New York NY, 10027; cg2597@columbia.edu} \and Fran\c cois Monard\thanks{Department of Mathematics, University of Washington, Seattle WA, 98195; fmonard@uw.edu}}

\begin{document}
\maketitle
\begin{abstract}This paper concerns the reconstruction of an anisotropic conductivity tensor $\gamma$ from internal current densities of the form $J=\gamma\nabla u$, where $u$ solves a second-order elliptic equation $\nabla\cdot(\gamma\nabla u)=0$ on a bounded domain $X$ with prescribed boundary conditions.
    A minimum number of such functionals equal to $n+2$, where $n$ is the spatial dimension, is sufficient to guarantee a local reconstruction.  We show that $\gamma$ can be uniquely reconstructed with a loss of one derivative compared to errors in the measurement of $J$. In the special case where $\gamma$ is scalar, it can be reconstructed with no loss of derivatives.  We provide a precise statement of what components  may be reconstructed with a loss of zero or one derivatives.

\end{abstract}
\section{Introduction}
Hybrid medical imaging modalities are extensively studied in the bio-engineering community. Such methods aim to combine high-contrast, such as the one found in the modalities Electrical Impedance Tomography (EIT) or Optical Tomography (OT), with high-resolution, as is observed in the modalities Magnetic Resonance Imaging (MRI) or ultrasound. The high-contrast modality EIT aims to locate unhealthy tissues by reconstructing their electrical conductivity $\gamma$ from current boundary measurements. This leads to an inverse problem known as Calder\'{o}n's problem. Extensive studies have been made on uniqueness properties and reconstruction methods for this inverse problem \cite{Uhl-review-99}. Unfortunately, the problem is severely ill-posed and yields images with poor resolution. 

It is sometimes possible to leverage a physical coupling between a high-contrast, low-resolution modality and a high-resolution, low-contrast modality. Such a coupling typically provides internal functionals of the unknown coefficients of interest and greatly improve its resolution \cite{Ammari2008,AS-IP-12,Bal2012c,Bal2010d,Bal2010,Monard2012b,S-SP-2011,Stefanov2012}. Different types of internal functionals, such as \emph{current densities} and \emph{power densities}, corresponding to different physical couplings have been analyzed to recover the unknown conductivity. In the case of power densities, we refer the reader to, e.g., \cite{Bal2011a,Bal2012e,Capdeboscq2009,Kuchment2011a,Kuchment2011,Monard2012b,Monard2012a,Monard2011,Monard2011a}. 

In this paper, we consider the Current Density Impedance Imaging problem (CDII), also called Magnetic Resonance Electrical Impedance Tomography (MREIT) of reconstructing an anisotropic conductivity tensor in the second-order elliptic equation,
\begin{align}
    \nabla\cdot(\gamma\nabla u) = \sum_{i,j=1}^{n}\partial_i(\gamma^{ij}\partial_j u)=0 \quad (X), \qquad u|_{\partial X} = g,
    \label{eq:conductivity}
\end{align}
from knowledge of internal current densities of the form $H=\gamma\nabla u$, where $u$ solves \eqref{eq:conductivity}. To be consistent with earlier publications, where the notation $H$ is used systematically to denote internal functionals, we use $H$ to denote current densities rather than the more customary notation $J$. Here $X$ is an open bounded domain with a $\C^{2,\alpha}$ or smoother boundary $\partial X$. The above equation has real-valued coefficients and $\gamma$ is a symmetric tensor satisfying the uniform ellipticity condition 
\begin{align}
    \kappa^{-1}\|\xi\|^2\le \xi\cdot\gamma\xi \le \kappa\|\xi\|^2, \quad \xi\in \Rm^n, \quad \text{for some } \kappa\ge 1,
    \label{positive definite}
\end{align}
so that \eqref{eq:conductivity} admits a unique solution in $H^1(X)$ for $g\in H^{\frac{1}{2}}(\partial X)$. 

Internal current density functionals $H$ can be obtained by the technique of current density imaging. The idea is to use Magnetic Resonance Imaging (MRI) to determine the magnetic field $B$ induced by an input current $I$. The current density is then defined by $H=\nabla \times B$.  We thus need to measure all components of $B$ to calculate $H$, which may create some difficulties in practice, but this is the starting point of this paper. See \cite{Ider1997, Scott1991} for details.

A perturbation method to reconstruct the unknown conductivity in the linearized case was presented in \cite{Ider1998}. In dimension $n=2$, a numerical reconstruction algorithm based on the construction of equipotential lines was given in \cite{Kwon2002}. Kwon \emph{et al} \cite{Kwon2002a} proposed a \emph{J}-substitution algorithm, which is an iterative algorithm. Assuming knowledge of only the magnitude of only one current density $|H|=|\gamma\nabla u|$, the problem was studied in \cite{Nachman2007,Nachman2009,NTT-Rev-11} (see the latter reference for a review) in the isotropic case and more recently in \cite{HMN2013,Ma2013} in the anisotropic case with anisotropy known. In \cite{Joy2004,Lee2004}, Nachman {\it et al.} and Lee independently found a explicit reconstruction formula for visualizing $\log\gamma$ at each point in a domain. The reconstruction with functionals of the form $\gamma^t\nabla u$ is shown in \cite{kocyigit2012} in the isotropic case. For $t=0$, the functionals are given by solutions of \eqref{eq:conductivity}, then a more general complex-valued tensor in the anisotropic case was presented in \cite{Bal2012}. 
In \cite{Seo2012}, assuming that the magnetic field $B$ is measurable, Seo \emph{et al.} gave a reconstruction for a complex-valued coefficient in the isotropic case.

\medskip

In the present work, we study the inverse problem in the anisotropic setting with a set of current densities $H_j=\gamma\nabla u_j$ for $1\leq j\leq m$, where $u_j$ solves \eqref{eq:conductivity} with prescribed boundary conditions $g_j$. We propose sufficient conditions on $m$ and the choice of $\{g_j\}_{\leq j\leq m}$ such that the reconstruction of $\gamma$ is unique and satisfies elliptic stability estimates. 


\section{Statement of the main results}\label{main results}

For $X\subset \Rm^n$, we denote by $\Sigma(X)$ the set of conductivity tensors with bounded components satisfying the uniform ellipticity condition \eqref{positive definite}. Then for $k\ge 1$ an integer and $0<\alpha<1$, we denote 
\begin{align*}
    \C_\Sigma^{k,\alpha} (X) := \{ \gamma\in \Sigma(X)| \quad \gamma_{pq} \in \C^{k,\alpha} (X),\quad 1\le p\le q\le n\}.
\end{align*}
In what follows, by ``solution of \eqref{eq:conductivity}'' we may refer to the solution itself or the boundary condition that generates it, i.e. $g = u|_{\partial X} \in H^{\frac{1}{2}}(\partial X)$. We will consider collections of measurements of the form 
\begin{align}
    H_i:\gamma\mapsto H_i(\gamma) = \gamma\nabla u_i,  \qquad 1\le i\le m,
    \label{eq:meas}
\end{align}
where $u_i$ solves \eqref{eq:conductivity} with boundary condition $g_i$. We decompose $\gamma$ into the product of a scalar factor $\beta$ with an anisotropic structure $\tilde{\gamma}$
\begin{align}
    \gamma:=\beta\tilde{\gamma}, \qquad \beta=(\det\gamma)^{\frac{1}{n}}, \qquad \det\tilde{\gamma}=1.
    \label{eq:decomp}
\end{align}
Since $\gamma$ satisfies the uniform elliptic condition \eqref{positive definite}, $\beta$ is bounded away from zero. 

From knowledge of a sufficiently large number of current densities, the reconstruction formulas for $\beta$ and $\tilde\gamma$ can be locally established in terms of the current densities and their derivatives up to first order. 

\subsection{Main hypotheses} \label{main hypotheses}

We begin with the main hypotheses that allow us to setup a few reconstruction procedures.

The first hypothesis aims at making the scalar factor $\beta$ in \eqref{eq:decomp} locally reconstructible via a gradient equation.
\begin{hypothesis} \label{2 sol} 
    There exist two solutions $(u_1, u_2)$ of \eqref{eq:conductivity} and $X_0\subset X$ convex satisfying
    \begin{align}
	\inf_{x\in X_0} \F_1(u_1,u_2) \ge c_0 >0 \where\quad  \F_1 (u_1,u_2):= |\nabla u_1|^2 |\nabla u_2|^2 - (\nabla u_1\cdot\nabla u_2)^2.
	\label{eq:F1}
    \end{align}
\end{hypothesis}

On to the hypotheses for local reconstructibility of $\tilde\gamma$, we first need to have, locally, a basis of gradients of solutions of \eqref{eq:conductivity}.

\begin{hypothesis} \label{main hypo} 
    There exist $n$ solutions $(u_1,\dots,u_n)$ of \eqref{eq:conductivity} and $X_0\subset X$ satisfying
    \begin{align}
	\inf_{x\in X_0} \F_2(u_1,\dots,u_n) \ge c_0 >0, \where \quad \F_2(u_1,\dots,u_n) := \det (\nabla u_1,\dots,\nabla u_n).
	\label{eq:F2}
    \end{align}
\end{hypothesis}

Let us now pick ${u_1,\cdots,u_n}$ satisfying Hyp. \ref{main hypo} and consider additional solutions $\{u_{n+k}\}_{k=1}^m$. Each additional solution decomposes in the basis $(\nabla u_1,\dots,\nabla u_n)$ as 
\begin{align}
    \nabla u_{n+k}=\sum_{i=1}^{n}\mu_k^i \nabla u_i, \quad  1\le k\le m,
    \label{ln dep}
\end{align}
where, as shown in \cite{Bal2012e} for instance, the coefficients $\mu_k^i$ take the expression 
\begin{align*}
    \mu_k^i = - \frac{\det (\nabla u_1,\dots,\overbrace{\nabla u_{n+k}}^i, \dots, \nabla u_n)}{\det (\nabla u_1,\dots,\nabla u_n)} = - \frac{\det (H_1,\dots,\overbrace{H_{n+k}}^i, \dots, H_n)}{\det (H_1,\dots,H_n)},
\end{align*}
in particular, these coefficients are {\em accessible from current densities}. The subsequent algorithms will make extensive use of the matrix-valued quantities
\begin{align}
    Z_k=\left[Z_{k,1} | \cdots | Z_{k,n}\right],\quad \text{where }\quad Z_{k,i} := \nabla \mu_k^i, \quad 1\le k\le m
    \label{Y Z}
\end{align}
In particular, the next hypothesis, formulating a sufficient condition for local reconstructibility of the anisotropic part of $\gamma$ is that, locally, a certain number of matrices $Z_{k}$ (at least two) satisfies some rank maximality condition. 

\begin{hypothesis} \label{hyp W}
    Assume that Hypothesis \ref{main hypo} holds for some $(u_1,\dots,u_n)$ over $X_0\subset X$ and denote by $H$ the matrix with columns $H_1, \dots, H_n$. Then there exist $u_{n+1}, \dots, u_{n+m}$ solutions of \eqref{eq:conductivity} and some $X'\subseteq X_0$ such that the $x$-dependent space
    \begin{align}
	\W :=  \text{span} \left\{ (Z_k H^T \Omega)^{sym}, \quad \Omega\in A_n(\Rm), 1\le k\le m \right\} \subset S_n(\Rm)
	\label{eq:W}
    \end{align}
    has codimension one in $S_n(\Rm)$ throughout $X'$. 
\end{hypothesis}

An alternate approach to reconstruct $\gamma$ is to set up a coupled system for $u_1,\dots,u_n$ satisfying Hyp. \ref{main hypo} globally. This system of PDEs can be derived under the following hypothesis (part A). From this system and under an additional hypothesis (part B), we can derive an elliptic system from which to reconstruct $u_1,\dots,u_n$.
\begin{hypothesis} \label{hy:full rank}
    \begin{itemize}
	\item[A.] Suppose that Hypothesis \ref{main hypo} is satisfied over $X_0=X$ for some solutions $(u_1,\dots,u_n)$. There exists an additional solution $u_{n+1}$ of \eqref{eq:conductivity} whose matrix $Z_1$ defined by \eqref{Y Z} is uniformly invertible over $X$, i.e. 
	    \begin{align}
		\inf_{x\in X} \det Z_1 \ge c_0 >0,
		\label{detZ1}
	    \end{align}
	    for some positive constant $c_0$. 
	\item[B.] There exist $n+2$ solutions $u_1,\dots,u_{n+2}$ such that $(u_1,\dots,u_n, u_{n+2})$ satisfy (A), and two $A_n(\Rm)$-valued functions $\Omega_1(x)$, $\Omega_2(x)$ such that the matrix
	    \begin{align}
		S=\left( Z_2^\star Z_1^T \Omega_1(x)+HZ_1^T \Omega_2(x) \right)^{sym} \qquad (\text{with }Z_2^\star:= Z_2^{-T})
		\label{eq:S}
	    \end{align}
	    satisfies the ellipticity condition \eqref{positive definite}.
    \end{itemize}    
\end{hypothesis}

The first important result to note is that the hypotheses stated above remain satisfied under some perturbations of the boundary conditions or the conductivity tensor for smooth enough topologies. 

\begin{proposition} \label{prop:perturb}
  Assume that Hypothesis \ref{2 sol}, \ref{main hypo}, \ref{hyp W} or \ref{hy:full rank} holds over some $X_0\subseteq X$ for a given number $m$ of solutions of \eqref{eq:conductivity} with boundary conditions $g_1,\dots,g_m$. Then for any $0<\alpha<1$, there exists a neighborhood of $(g_1,\dots,g_m,\gamma)$ open for the $\C^{2,\alpha}(\partial X)^m \times \C^{1,\alpha}(X)$ topology where the same hypothesis holds over $X_0$. In the case of \ref{hy:full rank}.B, it still holds with the same $A_n(\Rm)$-valued functions $\Omega_1$ and $\Omega_2$.
\end{proposition}

\subsection{Reconstruction algorithms and their properties}

\paragraph{Reconstruction of $\beta$ knowing $\tilde{\gamma}$.} Under knowledge of $\tilde\gamma$ and using two measurements $H_1,H_2$ coming from two solutions satisfying Hyp. \ref{2 sol} over some $X_0\subset X$, we can derive the following gradient equation for $\log\beta$ 
\begin{align}
    \begin{split}
	\nabla\log\beta &= \frac{1}{D|H_1|^2} \left( |H_1|^2\ d(\tilde\gamma^{-1}H_1) - (H_1\cdot H_2)\ d(\tilde\gamma^{-1}H_2) \right) (\tilde\gamma H_1, \tilde\gamma H_2) \tilde\gamma^{-1} H_1 \\
	& \qquad - \frac{1}{|H_1|^2}\ d(\tilde\gamma^{-1} H_1) (\tilde\gamma H_1, \cdot), \qquad x\in X_0,
    \end{split}  
    \label{log beta}
\end{align}
where $D := |H_1|^2 |H_2|^2 - (H_1\cdot H_2)^2$ is bounded away from zero over $X_0$ thanks to Hyp. \ref{2 sol}, and where the exterior calculus notations used here are recalled in Appendix \ref{app}.

Equation \eqref{log beta} allows us to reconstruct $\beta$ under the knowledge of $\beta(x_0)$ at one fixed point in $X_0$ by integrating \eqref{log beta} over any curve starting from some $x_0\in X_0$. This leads to a unique and stable reconstruction with no loss of derivatives, as formulated in the following proposition. This generalizes the result in \cite{Joy2004} to an anisotropic tensor.

\begin{proposition}[Local uniqueness and stability for $\beta$]\label{tilde gamma known}
    Consider two tensors $\gamma=\beta\tilde\gamma$ and $\gamma' = \beta'\tilde\gamma'$, where $\tilde\gamma, \tilde\gamma' \in W^{1,\infty}(X)$ are known. Suppose that Hypothesis \ref{2 sol} holds over the same $X_0\subset X$ for two pairs $(u_1,u_2)$ and $(u'_1,u'_2)$, solutions of \eqref{eq:conductivity} with conductivity $\gamma$ and $\gamma'$, respectively. Then the following stability estimate holds for any $p\ge 1$
    \begin{align}\label{stability beta}
	\|\log\beta-\log\beta'\|_{W^{p,\infty}(X_0)}\le \epsilon_0+C \left( \sum_{i=1,2} \|H_i-H'_i\|_{W^{p,\infty}(X)}+\|\tilde\gamma-\tilde\gamma'\|_{W^{p,\infty}(X)} \right)
    \end{align}
    Where $\epsilon_0=|\log\beta(x_0)-\log\beta'(x_0)|$ is the error committed at some fixed $x_0\in X_0$.
\end{proposition}


\paragraph{Algebraic, local reconstruction of $\tilde{\gamma}$:} On to the local reconstruction of the anisotropic structure, we start from $n+m$ solutions $(u_1,\dots,u_{n+m})$ satisfying hypotheses \ref{main hypo} and \ref{hyp W} over some $X_0\subset X$. In particular, the linear space $\W\subset S_n(\Rm)$ defined in \eqref{eq:W} is of codimension one in $S_n(\Rm)$. We will see that the tensor $\tilde\gamma$ must be orthogonal to $\W$ for the inner product $\langle A,B\rangle:=A_{ij}B_{ij}=\tr(AB^T)$. Together with the conditions that $\det \tilde\gamma = 1$ and $\tilde\gamma$ is positive, the space $\W$, known from the measurements $H_1, \dots, H_{n+m}$ completely determines $\tilde\gamma$ over $X_0$. In light of these observations, a constructive reconstruction algorithm based on a generalization of the cross-product is proposed in section \ref{sb tilde}. This approach was recently used in \cite{Monard2012a} in the context of inverse conductivity from power densities. This algorithm leads to a unique and stable reconstruction in the sense of the following proposition.

\begin{proposition}[Local uniqueness and stability for $\tilde\gamma$] \label{prop:gammatilde}
    Consider two uniformly elliptic tensors $\gamma$ and $\gamma'$. Suppose that Hypotheses \ref{main hypo} and \ref{hyp W} hold over the same $X_0\subset X$ for two $n+m$-tuples $\{u_i\}_{i=1}^{n+m}$ and $\{u'_i\}_{i=1}^{n+m}$, solutions of \eqref{eq:conductivity} with conductivity $\gamma$ and $\gamma'$, respectively. Then the following stability estimate holds for any integer $p\ge 0$
    \begin{align}
	\|\tilde\gamma - \tilde\gamma'\|_{W^{p,\infty}(X_0)}\le C \sum_{i=1}^{n+m} \|H_i-H'_i\|_{W^{p+1,\infty}(X)}.
	\label{eq:stabgammatilde}
    \end{align}
\end{proposition}

\paragraph{Joint reconstruction of $(\tilde\gamma,\beta)$, stability improvement for $\nabla\times\gamma^{-1}$.}

Judging by the stability estimates \eqref{eq:stabgammatilde} and \eqref{stability beta}, reconstructing $\beta$ after having reconstructed $\tilde\gamma$ is less stable (with respect to current densities) than when knowing $\tilde\gamma$. This is because in the former case, errors on $W^{p,\infty}$-norm in $\tilde\gamma$ are controlled by errors in $W^{p+1,\infty}$ norm in current densities. In particular, on the $W^{p,\infty}$ scale, stability on $\beta$ is no better than that of $\tilde\gamma$, and joint reconstruction of $(\tilde\gamma, \beta)$ using the preceding two algorithms displays the following stability, with $\gamma = \beta\tilde\gamma$
\begin{align}
    \|\gamma-\gamma'\|_{W^{p,\infty}(X_0)} \le C \sum_{i=1}^{n+m} \|H_i-H'_i\|_{W^{p+1,\infty}(X)}.
    \label{eq:jointstab}
\end{align}
However, once $\gamma$ is reconstructed, some linear combinations of first-order partials of $\gamma^{-1}$ can be reconstructed with better stability. These are the exterior derivatives of the columns of $\gamma^{-1}$, a collection of $n^2(n-1)/2$ scalar functions which we denote $\nabla\times\gamma^{-1}$ and is reconstructed via the formula
\begin{align}
    \partial_q \gamma^{pl} - \partial_p \gamma^{ql} = H^{il}(\gamma^{qj}\partial_p H_{ji}-\gamma^{pj}\partial_q H_{ji}), \quad 1\le l \le n, \quad 1\le p<q\le n,
    \label{tensor coefficient}
\end{align}
derived in Sec. \ref{sec:stabimprov} and assuming that we are working with a basis of solutions satisfying Hypothesis \ref{main hypo}. The stability statement \eqref{eq:jointstab} is thus somewhat improved into a statement of the form 
\begin{align}
    \|\gamma-\gamma'\|_{W^{p,\infty}(X_0)} + \|\nabla\times(\gamma^{-1}-\gamma^{'-1})\|_{W^{p,\infty}(X_0)}  \le C \sum_{i=1}^{n+m} \|H_i-H'_i\|_{W^{p+1,\infty}(X)},
    \label{eq:jointstab2}
\end{align}
where we have defined
\begin{align*}
    \|\nabla\times(\gamma^{-1}-\gamma^{'-1})\|_{W^{p,\infty}(X_0)} := \sum_{l=1}^n \sum_{1\le i<j\le n} \|\partial_j \gamma^{il} - \partial_i \gamma^{jl}\|_{W^{p,\infty}(X_0)}.
\end{align*}

\paragraph{Global reconstruction of $\gamma$ via a coupled elliptic system.}
While the preceding approach required a certain number of additional solutions, we now show how one can setup an alternate reconstruction procedure with only $m=2$ additional solutions satisfying Hyp. \ref{hy:full rank}. A microlocal study of linearized current densities functionals shows that this is the minimum number of functionals necessary to reconstruct all of $\gamma$. 

The present approach consists is eliminating $\gamma$ from the equations and writing an elliptic system of equations for the solutions $u_j$; see \cite{Bal2012e,Monard2012b,Monard2012a} for similar approaches in the setting of power density functionals. The method goes as follows. Assume that Hypothesis \ref{main hypo} holds for some $(u_1,\dots,u_n)$ over $X_0=X$ and denote $\left[\nabla U\right]=\left[\nabla u_1,\cdots,\nabla u_n\right]$ as well as $H=\left[H_1,\cdots,H_n\right]$. Since $H = \gamma[\nabla U]$, we can thus reconstruct $\gamma$ by $\gamma=[\nabla U]^{-1}H$ once $[\nabla U]$ is known. We now show that we may reconstruct $[\nabla U]$ by solving a second-order elliptic system of partial differential equations.

When Hyp. \ref{hy:full rank}.A is satisfied for some $u_{n+1}$ and considering an additional solution $u_{n+2}$ and its corresponding current density, we first derive a system of coupled partial differential equations for $(u_1,\dots,u_n)$, whose coefficients only depend on measured quantities. 
\begin{proposition} \label{th:couple system nl}
    Suppose $n+2$ solutions $(u_1,\dots,u_{n+2})$ satisfy Hypotheses \ref{main hypo} and \ref{hy:full rank}.A and consider their corresponding measurements $H_I = \{H_i\}_{i=1}^{n+2}$. Then the solutions $(u_1,\cdots,u_n)$ satisfy the coupled system of PDE's
    \begin{align}
	\begin{split}
	    Z_2^\star Z_1^T(\bfe_p\otimes \bfe_q-\bfe_q\otimes \bfe_p):\nabla^2 u_j+v^{pq}_{ij}\cdot\nabla u_i&=0,\\
	    HZ_1^T(\bfe_p\otimes \bfe_q-\bfe_q\otimes \bfe_p):\nabla^2 u_j+\tilde v^{pq}_{ij}\cdot\nabla u_i&=0, \qquad u_j|_{\partial X}=g_j,
	\end{split}
	\label{nl coupled}
    \end{align}
    for $1\le j\le n$ and $1\le p < q\le n$, and where the vector fields $\{v_{ij}^{pq}, \tilde v^{pq}_{ij}\}$ only depend on the current densities $H_I$.
\end{proposition}

If additionally, $u_{n+2}$ is such that Hyp. \ref{hy:full rank}.B is satisfied, we can deduce a strongly coupled elliptic system for $(u_1,\dots,u_n)$ from \eqref{nl coupled}.

\begin{theorem}\label{elliptic sys nl}
    With the hypotheses of Proposition \ref{th:couple system nl}, assume further that Hypothesis \ref{hy:full rank}.B holds for some $A_n(\Rm)$-valued functions 
    \begin{align*}
	    \Omega_i (x) = \sum_{1\le p<q\le n} \omega^{i}_{pq}(x) ( \bfe_p\otimes \bfe_q-\bfe_q\otimes \bfe_q), \quad i=1,2.
    \end{align*}
    Then $(u_1,\cdots,u_n)$ can be reconstructed via the strongly coupled elliptic system 
        \begin{align}
	-\nabla\cdot(S\nabla u_j)+W_{ij}\cdot\nabla u_i = 0, \quad u_j|_{\partial X}=g_j,  \quad 1\le j\le n,
	\label{elliptic sys}
    \end{align}
    where $S = \left( Z_2^\star Z_1^T \Omega_1(x)+HZ_1^T \Omega_2(x) \right)^{sym}$ as in \eqref{eq:S} and where we have defined 
    \begin{align}
	W_{ij} := \nabla\cdot S - \sum_{1\le p<q\le n} \omega^1_{pq}(x) v_{ij}^{pq}+\omega^2_{pq}(x) \tilde v_{ij}^{pq}, \quad 1\le i,j\le n.
	\label{eq:Wij}
    \end{align}
    Moreover, if system \eqref{elliptic sys} with trivial boundary conditions has only the trivial solution, $u_1,\dots,u_n$ are uniquely reconstructed. Subsequently, $\gamma$ reconstructed as $\gamma = H [\nabla U]^{-1}$ satisfies the stability estimate
    \begin{align}
	\| \gamma - \gamma' \|_{L^2(X)} + \| \nabla\times (\gamma^{-1} -\gamma^{'-1})  \|_{L^2(X)}  \leq C \|H_I-H'_I\|_{H^1(X)},
    \end{align}
    for data sets $H_I, H_I$ close enough in $H^1$-norm.
\end{theorem}

\subsection{What tensors are reconstructible ?}
We now conclude with a discussion regarding what tensors are reconstructible from current densities, based on the extent to which Hypotheses \ref{2 sol}-\ref{hy:full rank} can be fulfilled, so that the above reconstruction algorithms can be implemented.

\paragraph{Test cases.} 

\begin{proposition}\label{prop:identity}
    For any smooth domain $X\subset\Rm^n$ and considering a constant conductivity tensor $\gamma_0$, there exists a non-empty $\C^{2,\alpha}$-open subset of $[H^{\frac{1}{2}}(\partial X)]^{n+2}$ of boundary conditions fulfilling Hypotheses \ref{2 sol}-\ref{hy:full rank} throughout $X$.     
\end{proposition}

The second test case regards isotropic smooth tensors of the form $\gamma = \beta \Imm_n$, where we show that the scalar coefficient $\beta$ can be reconstructed globally by using the real and imaginary parts of the same complex geometrical optics (CGO) solution. The use of CGOs for fulfilling internal conditions was previously used in \cite{Bal2011a,Bal2012,Monard2011a}.

\begin{proposition}\label{prop:isotropic}
    For an isotropic tensor $\gamma = \beta \Imm_n$ with $\beta\in H^{\frac{n}{2}+3+\varepsilon} (X)$ for some $\varepsilon>0$, there exists a non-empty $\C^{2,\alpha}$-open subset of $[H^{\frac{1}{2}}(\partial X)]^2$ fulfilling Hypothesis \ref{2 sol} thoughout $X$.    
\end{proposition}

Thanks to Proposition \ref{prop:perturb}, we can also formulate the following without proof. 
\begin{corollary}\label{cor:perturb}
    Suppose $\gamma$ is a tensor as in either Proposition \ref{prop:identity} or \ref{prop:isotropic}. Then, for any $0<\alpha< 1$, there exists a $\C^{1,\alpha}$-neighborhood of $\gamma$ for which the conclusion of the same proposition remains valid. 
\end{corollary}

\paragraph{Push-forwards by diffeomorphisms}

Recall that for $\Psi:X\to \Psi(X)$ a $W^{1,2}$-diffeomorphism and $\gamma\in \Sigma(X)$, we define $\Psi_\star \gamma$ the conductivity tensor push-forwarded by $\Psi$ from $\gamma$ defined over $\Psi(X)$, by
\begin{align}
    \Psi_\star \gamma := (|J_\Psi|^{-1} D\Psi\cdot\gamma\cdot D\Psi)\circ \Psi^{-1}, \qquad J_\Psi := \det D\Psi.
    \label{eq:pushfwd}
\end{align}
We now show that, whenever a tensor is being push-forwarded from another by a diffeormorphism, then the local or global reconstructibility of one is equivalent to that of the other, in the sense of the Proposition below. While the existence of $\Psi_\star \gamma$ in $\Sigma (\Psi(X))$ merely requires that $\Psi$ be a $W^{1,2}$-diffeomorphism, our results below will require that $\Psi$ be smoother and that it satisfies the following uniform condition over $X$
\begin{align}
    C_\Psi^{-1} \le |J_\Psi| \le C_\Psi \quad \text{for some } C_\Psi \ge 1. 
    \label{eq:psiOK}
\end{align}

\begin{proposition}\label{prop:pushfwd}
    Assume that Hypothesis \ref{2 sol}, \ref{main hypo}, \ref{hyp W} or \ref{hy:full rank} holds over some $X_0\subseteq X$ for a given number $m$ of solutions of \eqref{eq:conductivity} with boundary conditions $g_1,\dots,g_m$. For $\Psi:X\to\Psi(X)$ a smooth diffeomorphism satisfying \eqref{eq:psiOK}, the same hypothesis holds true over $\Psi(X_0)$ for the conductivity tensor $\Psi_\star \gamma$ with boundary conditions $(g_1\circ\Psi^{-1}, \dots, g_m\circ\Psi^{-1})$. In the case of Hyp. \ref{hy:full rank}.B, it holds with the following $A_n(\Rm)$-valued functions defined over $\Psi(X)$:
    \begin{align}
	\Psi_\star \Omega_1 := [D\Psi\cdot \Omega_1\cdot D\Psi^t]\circ \Psi^{-1} \qandq \Psi_\star \Omega_2 :=  [|J_\Psi| D\Psi\cdot \Omega_2\cdot D\Psi^t] \circ\Psi^{-1}.  
	\label{eq:psiomega}
    \end{align}
\end{proposition}

In contrast to inverse conductivity problems from boundary data, where the diffeomorphisms above are a well-known obstruction to injectivity, Proposition \ref{prop:pushfwd} precisely states the opposite: if a given tensor $\gamma$ is reconstructible in some sense, then so is $\Psi_\star \gamma$, and the boundary conditions making the inversion valid are explicitely given in terms of the ones that allow to reconstruct $\gamma$. 

\begin{corollary}\label{cor:pushfwd}
    Suppose $\gamma$ is a tensor as in either Proposition \ref{prop:identity} or \ref{prop:isotropic} and $\Psi:X\to \Psi(X)$ is a diffeomorphism satisfying \eqref{eq:psiOK}. Then the conclusion of the same proposition holds for the tensor $\Psi_\star \gamma$ over $\Psi(X)$ and  boundary conditions defined over $\partial (\Psi(X))$.     
\end{corollary}

\paragraph{Generic reconstructibility.}

We finally state that any $\C^{1,\alpha}$ smooth tensor is, in principle, reconstructible from current densities in the sense of the following proposition. This result uses the Runge approximation property, a property equivalent to the unique continuation principle, valid for Lipschitz-continuous tensors. 

\begin{proposition}\label{prop:Runge}
    Let $X\subset\Rm^n$ a $\C^{2,\alpha}$ domain and $\gamma\in \C^{1,\alpha}_\Sigma (X)$. Then for any $x_0\in X$, there exists a neighborhood $X_0\subset X$ of $x_0$ and $n+2$ solutions of \eqref{eq:conductivity} fulfilling hypotheses \ref{main hypo} and \ref{hyp W} over $X_0$. 
\end{proposition}


	    
    


\paragraph{Outline:} The rest of the paper is structured as follows. Section \ref{sec:prelim} covers the preliminaries, including the proof of Proposition \ref{prop:perturb}. Section \ref{sec:algos} presents the derivations of the local reconstruction algorithms: Sec. \ref{sb tau} covers the local reconstruction of $\beta$ and proves Proposition \ref{tilde gamma known}; Sec. \ref{sb tilde} covers the local reconstruction of $\tilde\gamma$ and the proof of Proposition \ref{prop:gammatilde}; Sec. \ref{sec:stabimprov} justifies equation \eqref{tensor coefficient}; Sec. \ref{sec:elliptic} discusses the global reconstruction of $\gamma$ via an elliptic system, with a proof of Propositions \ref{th:couple system nl} and \ref{elliptic sys nl}. Finally, Section \ref{sec:recons} discusses the question of reconstructibility from current densities, with the proofs of Propositions \ref{prop:identity}, \ref{prop:isotropic}, \ref{prop:pushfwd} and \ref{prop:Runge}.

\section{Preliminaries}\label{sec:prelim}
In this section, we briefly recall elliptic regularity results, the mapping properties of the current density operator and we conclude with the proof of Proposition \ref{prop:perturb}.

\paragraph{Properties of the forward mapping.}  In the following, we will make use of the following result, based on Schauder estimates for elliptic equations. It is for instance stated in \cite{Isakov2006}. 
\begin{proposition}\label{prop:fwdmap}
    For $k\ge 2$ an integer and $0<\alpha<1$, if $X$ is a $\C^{k+1,\alpha}$-smooth domain, then the mapping $(g,\gamma)\mapsto u$, solution of \eqref{eq:conductivity}, is continuous in the functional setting 
    \begin{align*}
	\C^{k,\alpha}(\partial X) \times \C_\Sigma^{k-1,\alpha} (X) \to \C^{k,\alpha}(X).
    \end{align*}    
\end{proposition}

As a consequence, we can claim that, with the same $k,\alpha$ as above, the current density operator $(g,\gamma)\mapsto \gamma\nabla u$ is continuous in the functional setting 
\begin{align*}
    \C^{k,\alpha}(\partial X) \times \C_\Sigma^{k-1,\alpha} (X) \to \C^{k-1,\alpha}(X).
\end{align*}
Moreover, this fact allows us to prove Proposition \ref{prop:perturb}.

\begin{proof}[Proof of Proposition \ref{prop:perturb}.] Fixing some domain $X_0\subset X$ and using Proposition \ref{prop:fwdmap}, it is clear that the mappings
    \begin{align*}
	f_1 : (\C^{2,\alpha}(\partial X))^2\times \C^{1,\alpha}_\Sigma(X)\ni  (g_1,g_2,\gamma) &\mapsto \inf_{X_0} \F_1(u_1,u_2), \\
	f_2 : (\C^{2,\alpha}(\partial X))^n\times \C^{1,\alpha}_\Sigma(X)\ni  (g_1,\dots,g_n,\gamma) &\mapsto \inf_{X_0} \F_2(u_1,\dots,u_n),
    \end{align*}
    with $\F_1,\F_2$ defined in \eqref{eq:F1},\eqref{eq:F2}, are continuous, so $f_1^{-1}\big( (0,\infty) \big)$ and $f_2^{-1}\big( (0,\infty) \big)$ are open, which takes care of Hypotheses \ref{2 sol} and \ref{main hypo}. Further, Hypothesis \ref{hyp W} is fulfilled if and only if condition \ref{noln det} holds. Again, using Prop. \ref{prop:fwdmap}, the mapping $f_3:= \inf_{X_0} \B$ with $\B$ defined in \eqref{noln det} is a continuous function of $(g_1,\dots,g_{n+m},\gamma)\in (\C^{2,\alpha}(\partial X))^{n+m}\times \C^{1,\alpha}_\Sigma(X)$ so that $f_3^{-1}\big( (0,\infty) \big)$ is open.

    Along the same lines, Hypothesis \ref{hy:full rank}.A is stable under such perturbations because the mapping 
    \begin{align*}
	(\C^{2,\alpha}(\partial X))^{n+1} \times \C^{1,\alpha}_\Sigma(X)\ni  (g_1,\dots,g_{n+1},\gamma) &\mapsto \inf_{X} \det Z_1,
    \end{align*}
    is continuous whenever $u_1,\dots,u_n$ satisfy \eqref{eq:F2} over $X$. Finally, fixing two $A_n(\Rm)$-valued functions $\Omega_1(x)$ and $\Omega_2(x)$, Hypothesis \ref{hy:full rank}.B is fulfilled whenever 
    \begin{align}
	(g_1,\dots,g_{n+2},\gamma) \in \bigcap_{i=1}^n s_i^{-1} \big( (0,\infty) \big),
	\label{eq:openset}
    \end{align}
    where we have defined the functionals, for $1\le i\le n$
    \begin{align*}
	s_i :  (\C^{2,\alpha}(\partial X))^{n+2} \times \C^{1,\alpha}_\Sigma(X)\ni  (g_1, \dots, g_{n+2},\gamma) \mapsto \inf_X \det \{S_{pq}\}_{1\le p,q\le i}, 
    \end{align*}
    with $S = \{S_{p,q}\}_{1\le p,q\le n}$ defined as in \eqref{eq:S}. Such functionals are, again, continuous, in particular the set in the right-hand side of \eqref{eq:openset} is open. This concludes the proof. 
\end{proof}

\section{Reconstruction approaches}\label{sec:algos}

\subsection{Local reconstruction of $\beta$} \label{sb tau}

In this section, we assume that $\tilde\gamma$ is known and with $W^{1,\infty}$ components. Assuming Hypothesis \ref{2 sol} is fulfilled for two solutions $u_1,u_2$ over an open set $X_0\subset X$, we now prove equation \eqref{log beta}. 

\begin{proof}[Proof of equation \eqref{log beta}.]
    Rewriting \eqref{eq:meas} as $\frac{1}{\beta}\tilde{\gamma}^{-1}H_j=\nabla u_j$ and applying the operator $d(\cdot)$. Using identities \eqref{eq:d2zero} and \eqref{eq:dfV}, we arrive at the following equation for $\log\beta$:
    \begin{align} \label{2 form nl}
	\nabla\log\beta \wedge (\tilde{\gamma}^{-1}H_j) = d(\tilde{\gamma}^{-1}H_j), \quad j=1,2.
    \end{align}
    Let us first notice the following equality of vector fields
    \begin{align*}
	\nabla\log\beta\wedge (\tilde\gamma^{-1}) H_1 (\tilde\gamma H_1, \cdot) = (\nabla\log\beta\cdot \tilde\gamma H_1) (\tilde\gamma^{-1} H_1) - |H_1|^2 \nabla\log\beta, 
    \end{align*}
    so that 
    \begin{align*}
	\nabla\log\beta &= \frac{1}{|H_1|^2} (\nabla\log\beta\cdot \tilde\gamma H_1) \tilde\gamma^{-1} H_1 - \frac{1}{|H_1|^2} \nabla\log\beta\wedge (\tilde\gamma^{-1} H_1) (\tilde\gamma H_1, \cdot) \\
	&= \frac{1}{|H_1|^2} (\nabla\log\beta\cdot \tilde\gamma H_1) \tilde\gamma^{-1} H_1 - \frac{1}{|H_1|^2} d(\tilde\gamma^{-1} H_1) (\tilde\gamma H_1, \cdot).
    \end{align*}
    It remains thus to prove that 
    \begin{align*}
	(\nabla\log\beta\cdot \tilde\gamma H_1) =  \frac{1}{D} \left( |H_1|^2 d(\tilde\gamma^{-1}H_1) - (H_1\cdot H_2) d(\tilde\gamma^{-1}H_2) \right) (\tilde\gamma H_1, \tilde\gamma H_2),
    \end{align*}
    which may be checked directly by computing, for $j=1,2$ 
    \begin{align*}
	d(\tilde\gamma^{-1} H_j) (\tilde\gamma H_1, \tilde\gamma H_2) &= d\log\beta\wedge (\tilde\gamma^{-1} H_j)  (\tilde\gamma H_1, \tilde\gamma H_2) \\
	& = (\nabla\log\beta\cdot \tilde\gamma H_1 ) H_j\cdot H_2 - (\nabla\log\beta\cdot \tilde\gamma H_2 )(H_j\cdot H_1). 
    \end{align*}
    Taking the appropriate weighted sum of the above equations allows to extract $(\nabla\log\beta\cdot \tilde\gamma H_1)$, and hence \eqref{log beta}.
\end{proof}

\paragraph{Reconstruction procedures for $\beta$, uniqueness and stability.}

Suppose equation \eqref{log beta} holds over some convex set $X_0\subset X$ and fix $x_0\in X_0$. Equation \eqref{log beta} is a gradient equation $\nabla\log\beta = F$ with known right-hand side $F$. For any $x\in X_0$, one may thus construct $\beta(x)$ by integrating \eqref{log beta} over the segment $[x_0,x]$, leading to one possible formula
\begin{align}
    \beta(x) = \beta(x_0) \exp \left( \int_0^1 (x-x_0)\cdot F( (1-t)x_0 + tx)\ dt \right), \quad x\in X_0. 
    \label{eq:reconsbeta}
\end{align}

\begin{proof}[Proof of Proposition \ref{tilde gamma known}]
    Since $\det\tilde\gamma = 1$, the entries of $\tilde\gamma^{-1}$ are polynomials of the entries of $\tilde\gamma$, so that the entries of the right-hand side of \eqref{log beta} are polynomials of the entries of $H_1,H_2,\tilde\gamma$ and their derivatives, with bounded coefficients. It is thus straightforward to establish that 
    \begin{align}\label{estimate nabla beta}
	\|\nabla\log\beta-\nabla\log\beta'\|_{L^{\infty}(X_0)}\le C(\|H-H'\|_{W^{1,\infty}(X)}+\|\tilde\gamma-\tilde\gamma'\|_{W^{1,\infty}(X)})
    \end{align}
    for some constant $C$. Estimate \eqref{stability beta} then follows from the fact that 
    \begin{align*}
	\|\log\beta-\log\beta'\|_{L^\infty(X_0)} \le |\log\beta(x_0)-\log\beta'(x_0)| + \Delta (X) \|\nabla\log\beta-\nabla\log\beta'\|_{L^{\infty}(X_0)},
    \end{align*}    
    where $\Delta (X)$ denotes the diameter of $X$. 
\end{proof}

One could use another integration curve than the segment $[x_0,x]$ to compute $\beta(x)$. In order for this integration to not depend on the choice of curve, the right-hand side $F$ of \eqref{log beta} should satisfy the integrability condition $dF = 0$, a condition on the measurements which characterizes partially the range of the measurement operator. 

When measurements are noisy, said right-hand side may no longer satisfy this requirement, in which case the solution to \eqref{log beta} no longer exists. One way to remedy this issue is to solve the {\em normal} equation to \eqref{log beta} over $X_0$ (whose boundary can be made smooth) with, for instance, Neuman boundary conditions:
\begin{align*}
    -\Delta \log\beta = -\nabla\cdot F \quad (X_0), \qquad \partial_\nu \log\beta|_{\partial X_0} = F\cdot \nu,
\end{align*}
where $\nu$ denotes the outward unit normal to $X_0$. This approach salvages existence while projecting the data onto the range of the measurement operator, with a stability estimate similar to \eqref{stability beta} on the $H^s$ Sobolev scale instead of the $W^{s,\infty}$ one.

\subsection{Local reconstruction of $\tilde\gamma$} \label{sb tilde}

We now turn to the local reconstruction algorithm of $\tilde\gamma$. In this case, the reconstruction is algebraic, i.e. no longer involves integration of a gradient equation. In the sequel, we work with $n+m$ solutions of \eqref{eq:conductivity} denoted $\{u_i\}_{i=1}^{n+m}$, whose current densities $\{H_i = \gamma\nabla u_i\}_{i=1}^{n+m}$ are assumed to be measured. 

\paragraph{Derivation of the space of linear constraints \eqref{eq:W}.} Apply the operator $d(\gamma^{-1}\cdot)$ to the relation of linear dependence
\begin{align*}
    H_{n+k} = \mu_k^i H_i, \where\quad \mu_k^i:= - \frac{\det (H_1,\dots,\overbrace{H_{n+k}}^i, \dots, H_n)}{\det (H_1,\dots,H_n)},\quad 1\le i\le n.
\end{align*}
Using the fact that $d(\gamma^{-1}H_i)=d (\nabla u_i)=0$, we arrive at the following relation,
\begin{align*}
    Z_{k,i}\wedge \tilde\gamma^{-1}H_i=0, \where\quad Z_{k,i}:=\nabla\mu_k^i, \quad k=1,2, \ldots
\end{align*}
Since the 2-form vanishes, by applying two vector fields $\tilde{\gamma}\bfe_p$, $\tilde{\gamma}\bfe_p$, $1\le p< q\le n$, we obtain,
\begin{align*}
    H_{qi}Z_{k,i}\cdot\tilde{\gamma}\bfe_p=H_{pi}Z_{k,i}\cdot\tilde{\gamma}\bfe_q,
\end{align*}
Notice that the above equation means $(\tilde\gamma Z_k)_{pi}H_{qi}=(\tilde\gamma Z_k)_{qi}H_{pi}$, which amounts to the fact that $\tilde\gamma Z_k H^T$ is symmetric. This means in particular that $\tilde\gamma Z_k H^T$ is orthogonal to $A_n(\Rm)$, and for any $\Omega\in A_n(\Rm)$, we can rewrite this orthogonality condition as
\begin{align}\label{symmetry nl}
    0 = \tr (\tilde\gamma Z_k H^T \Omega) = \tr (\tilde\gamma^T Z_k H^T \Omega) = \tilde\gamma:Z_kH^T\Omega = \tilde\gamma:(Z_kH^T\Omega)^{sym},
\end{align}
where the last part comes from the fact that $\tilde\gamma$ is itself symmetric. Each matrix $Z_k$ thus generates a subspace of $S_n(\Rm)$ of linear contraints for $\tilde\gamma$. Considering $m$ additional solutions, we arrive at the space of constraints defined in \eqref{eq:W}.

\paragraph{Algebraic inversion of $\tilde\gamma$ via cross-product.} We now show how to reconstruct $\tilde\gamma$ explicitely at any point where the space $\W$ defined in \eqref{eq:W} has codimension one. We define the generalized cross product as follows. Over an $N$-dimensional space $\V$ with a basis $(\bfe_1,\cdots,\bfe_N)$, we define the alternating $N-1$-linear mapping $\N: \V^{N-1}\to \V$ as the formal vector-valued determinant below, to be expanded along the last row
\begin{align} \label{cross product}
    \N(V_1,\cdots,V_{N-1}):=\frac{1}{\det(\bfe_1,\cdots,\bfe_N)}\left|
    \begin{array}{ccc}
	\langle V_1,\bfe_1\rangle & \ldots &\langle V_1,\bfe_N\rangle \\
	\vdots      & \ddots  & \vdots\\
	\langle V_{N-1},\bfe_1\rangle & \ldots & \langle V_{N-1},\bfe_N\rangle\\
	\bfe_1 &\ldots &\bfe_N
    \end{array}
    \right|
\end{align}
$\N (V_1,\cdots,V_{N-1})$ is orthogonal to $V_1,\cdots,V_{N-1}$. Moreover, $\N(V_1,\cdots,V_{N-1})$ vanishes if and only if $(V_1,\cdots,V_{N-1})$ are linearly dependent.

With this notion of cross-product in the case $\V \equiv S_n(\Rm)$, we derive the following reconstruction algorithm for $\tilde\gamma$. Adding $m$ additional solutions, we find that $\W$ can be spanned by $\sharp\W := \frac{n(n-1)}{2} m$ matrices whose expressions are given in \eqref{eq:W}, picking for instance $\{\bfe_i\otimes \bfe_j - \bfe_j\otimes \bfe_i\}_{1\le i<j\le n}$ as a basis for $A_n(\Rm)$. The condition that $\W$ is of codimension one over $X_0$ can be formulated as:
\begin{align}
    \inf_{x\in X_0} \B(x) >c_1>0 , \qquad \B:= \sum_{I\in\sigma(n_S-1,\sharp\W)}|\det\N(I)|^{\frac{1}{n}},
    \label{noln det}
\end{align}
where $\sigma(n_S-1,\sharp\W)$ denotes the sets of increasing injections from $[1,n_S-1]$ to $[1,\sharp\W]$, and where we have defined $\N(I)=\N(M_{I_1},\cdots,M_{I_{n_S-1}})$, where $\N$ is defined by \eqref{cross product} with $\V \equiv S_n(\Rm)$. Then under condition \eqref{noln det}, $\W$ is of rank $n_S-1$ in $S_n(\Rm)$. 

Whenever $(M_1,\dots,M_{n_S-1})$ are picked in $\W$, their cross-product must be proportional to $\tilde \gamma$. The constant of proportionality can be deduced, up to sign, from the condition $\det\tilde\gamma=1$ so we arrive at $\pm|\det \N(M_1,\cdots,M_{n_S-1})|^{\frac{1}{n}}\tilde{\gamma}=\N(M_1,\cdots,M_{n_S-1})$. The sign ambiguity is removed by ensuring that $\tilde\gamma$ must be symmetric definite positive, in particular its first coefficient on the diagonal should be positive. As a conclusion, we obtain the relation
\begin{align} 
    |\det\N(I)|^{\frac{1}{n}}\tilde\gamma=\text{sign}(\N_{11}(I))\N(I), \quad I \in \sigma (n_S-1, \sharp \W).
    \label{gamma tilde}
\end{align}
This relation is nontrivial (and allows to reconstruct $\tilde\gamma$) only if $(M_1,\dots,M_{n_S-1})$ are linearly independent. When $\text{codim } \W =1$ but $\sharp \W > n_S-1$, we do not know {\it a priori} which $n_S-1$ -subfamily of $\W$ has maximal rank, so we sum over all possibilities. Equation \eqref{gamma tilde} then becomes
\begin{align}
    \sum_{I\in\sigma(n_S-1,\sharp\W)}\text{sign}(\N_{11}(I))\N(I)=\B\tilde\gamma,
    \label{noln algebra}
\end{align}
with $\B$ defined in \eqref{noln det}. Since $\B>c_1>0$ over $X_0$, $\tilde\gamma$ can be algebraically reconstructed on $X_0$ by formula \eqref{noln algebra}, where $\N$ is defined by \eqref{cross product} with $\V=S_n(\Rm)$.

\paragraph{Uniqueness and stability.}

Formula \eqref{noln algebra} has no ambiguity provided condition \eqref{noln det}, hence the uniqueness. Regarding stability, we briefly justify Proposition \ref{prop:gammatilde}.

\begin{proof}[Proof of Proposition \ref{prop:gammatilde}.] In formula \eqref{noln algebra}, the components of the cross-products $\N(I)$ are smooth (polynomial) functions of the components of the matrices $Z_k H$, which in turn are smooth functions of the components of $\{H_i\}_{i=1}^{n+m}$ and their first derivatives, and where the only term appearing as denominator is $\det (H_1,\dots,H_n)$, which is bounded away from zero by virtue of Hypothesis \ref{main hypo}. Thus \eqref{eq:stabgammatilde} holds for $p=0$. That it holds for any $p\ge 1$ is obtained by taking partial derivatives of the reconstruction formula of order $p$ and bounding accordingly.  
\end{proof}

\subsection{Joint reconstruction of $(\tilde\gamma,\beta)$ and stability improvement}\label{sec:stabimprov}
In this section, we justify equation \eqref{tensor coefficient}, which allows to justify the stability claim \eqref{eq:jointstab2}. Starting from $n$ solutions satisfying Hypothesis \ref{main hypo} over $X_0\subseteq X$ and denote $H = \{H_{ij}\}_{i,j=1}^n = [H_1|\dots|H_n]$ as well as $H^{pq} := (H^{-1})_{pq}$. Applying the operator $d (\gamma^{-1}\cdot)$ to both sides of \eqref{eq:meas} yields $d(\gamma^{-1}H_j)= d(\nabla u_j) = 0$ due to \eqref{eq:d2zero}. Rewritten in scalar components for $1\le j\le n$ and $1\le p< q\le n$
\begin{align*}
    0 = \partial_q (\gamma^{pl} H_{lj} ) - \partial_p (\gamma^{ql} H_{lj}) = (\partial_q \gamma^{pl}-\partial_p\gamma^{ql} )H_{lj}+\gamma^{pl}\partial_q H_{lj}-\gamma^{ql}\partial_p H_{lj}.
\end{align*}
Thus \eqref{tensor coefficient} is obtained after multiplying the last right-hand side by $H^{ji}$, summing over $j$ and using the property that $\sum_{j=1}^n H_{lj}H^{ji} = \delta_{il}$. 

\subsection{Reconstruction of $\gamma$ via an elliptic system} \label{sec:elliptic}
In this section, we will construct a second order system for $(u_1,\cdots,u_n)$ with $n+2$ measurements, assuming Hypotheses \ref{main hypo} and \ref{hy:full rank}.A hold with $X_0=X$. For the proof below, we shall recall the definition of the Lie Bracket of two vector fields in the euclidean setting:
\begin{align*}
    \left[X,Y\right] := (X\cdot\nabla)Y-(Y\cdot\nabla)X = (X^i \partial_i) Y^j \bfe_j - (Y^i \partial_i) X^j \bfe_j.
\end{align*}

\begin{proof}[Proof of Proposition \ref{th:couple system nl}]
    As is shown by \eqref{symmetry nl}, $\gamma Z_k H^T$ is symmetric. Multiplying both sides by $\gamma^{-1}$ and using $\gamma^{-1}H=\nabla U$, we see that $Z_k[\nabla U]^T$ is symmetric. More explicitly, we have 
    \begin{align}\label{Z U nl}
	Z_{k,pi}\partial_q u_i=Z_{k,qi}\partial_p u_i, \quad k=1,2,
    \end{align}
    or simply $Z_k[\nabla U]^T=[\nabla U]Z_k^T$. Assume Hypothesis \ref{hy:full rank}.A holds with $Z_2$ invertible so that $(Z_{2,1},\cdots,Z_{2,n})$ form a basis in $\Rm^n$. We define its dual frame such that $Z_{2,j}^{\star}\cdot Z_{2,i}=\delta _{ij}$. Denote $Z_2^{\star}=[Z_{2,1}^{\star},\cdots,Z_{2,n}^{\star}]$ and $Z_2^{\star}=Z_2^{-T}$. Then the symmetry of $Z_2[\nabla U]^T$ reads,
    \begin{align}\label{Y star nl}
	Z_{2,j}^\star\cdot\nabla u_i = Z_{2,i}^\star\cdot\nabla u_j, \quad 1\le i\le j\le n.
    \end{align}
    Pick $v$ a scalar function, we have the following commutation relation:
    \begin{align*}
	(X\cdot\nabla)(Y\cdot\nabla)v=(Y\cdot\nabla)(X\cdot\nabla)v+\left[X,Y\right]\cdot\nabla v
    \end{align*}
    Rewrite $Z_{1,pi}\partial_q=Z_{1,pi}\bfe_q\cdot\nabla$ and apply $Z_{2,j}^\star\cdot\nabla$ to both sides of \eqref{Z U nl}, we have the following equation by the above relations in Lie Bracket,
    \begin{align}\label{Y star Z}
	\left[Z_{2,j}^\star,Z_{1,pi}\bfe_q\right]\cdot \nabla u_i+(Z_{1,pi}\bfe_q\cdot\nabla)(Z_{2,j}^\star\cdot\nabla)u_i=\left[Z_{2,j}^\star,Z_{1,qi}\bfe_p\right]\cdot \nabla u_i+(Z_{1,qi}\bfe_p\cdot\nabla)(Z_{2,j}^\star\cdot\nabla)u_i
    \end{align}
    where $Z_{k,ij}=Z_k:\bfe_i\otimes\bfe_j$. Plugging \eqref{Y star nl} to the above equation gives,
    \begin{align*}
	(Z_{1,pi}\bfe_q\cdot\nabla)(Z_{2,i}^\star\cdot\nabla)u_j+\left[Z_{2,j}^\star,Z_{1,pi}\bfe_q\right]\cdot \nabla u_i=(Z_{1,qi}\bfe_p\cdot\nabla)(Z_{2,i}^\star\cdot\nabla)u_j+\left[Z_{2,j}^\star,Z_{1,qi}\bfe_p\right]\cdot \nabla u_i
    \end{align*}
    Looking at the principal part, the first term of the LHS reads
    \begin{align*}
	(Z_{1,pi}\bfe_q\cdot\nabla)(Z_{2,i}^\star\cdot\nabla)u_j=(Z_2^\star Z_1^T \bfe_p\otimes \bfe_q):\nabla^2u_j+(Z_{1,pi}\bfe_q\cdot\nabla)Z_{2,i}^\star\cdot \nabla u_j.
    \end{align*}
    Therefore, \eqref{Y star Z} amounts to the following coupled system,
    \begin{align}\label{prop Z2}
	Z_2^\star Z_1^T(\bfe_p\otimes \bfe_q-\bfe_q\otimes \bfe_p):\nabla^2 u_j+v^{pq}_{ij}\cdot\nabla u_i=0,\quad u_j|_{\partial X}=g_j,  \quad 1\le p\le q\le n
    \end{align}
    where
    \begin{align}\label{prop vij}
	v^{pq}_{ij} &:=\delta_{ij}\left[(Z_{1,pl}\bfe_q-Z_{1,ql}\bfe_p)\cdot\nabla\right]Z_{2,l}^\star+\left[Z_{2,j}^\star,Z_{1,pi}\bfe_q-Z_{1,qi}\bfe_p\right]. 
    \end{align}
    Notice that $H=\gamma[\nabla U]$ implies that $H^{-T}[\nabla U]^T$ is symmetric. Compared with equation \eqref{Z U nl}, we can see that the same proof holds if we replace $Z_2$ by $H^{-T}$. In this case, the dual frame of $H^{-T}$ is simply $H$. So \eqref{prop Z2} and \eqref{prop vij} hold by replacing $Z_2^{\star}$ by $H$ and defining $\tilde v_{ij}^{pq}$ accordingly. 
\end{proof}

We now suppose that Hypothesis \ref{hy:full rank}.B is satisfied and proceed to the proof of Theorem \ref{elliptic sys nl}.  

\begin{proof} Starting from Hypothesis \ref{hy:full rank}.B with $A_n(\Rm)$-valued functions of the form
    \begin{align*}
	    \Omega_i (x) = \sum_{1\le p<q\le n} \omega^{i}_{pq}(x) ( \bfe_p\otimes \bfe_q-\bfe_q\otimes \bfe_q), \quad i=1,2,
    \end{align*}
    we take the weighted sum of equations \eqref{nl coupled} with weights $\omega_{pq}^1,\omega_{pq}^2$. The principal part becomes $S:\nabla^2 u_i$, which upon rewritting it as $\nabla\cdot(S\nabla u_i) - (\nabla\cdot S)\cdot \nabla u_i$ yields system \eqref{elliptic sys}.

    On to the proof of stability, pick another set of data $H'_I:=\{H'_i\}_{i=1}^{n+2}$ close enough to $H_I$ in $W^{1,\infty}$ norm, and write the corresponding system for $u'_1,\dots,u'_n$
    \begin{align} \label{elliptic nl 2}
	-\nabla\cdot S'\nabla u'_j+W'_{ij}\cdot\nabla u'_i = 0, \quad 1\le j\le n,
    \end{align}
    where $S'$ and $W'_{ij}$ are defined by replacing $H_I$ in \eqref{eq:Wij} by $H'_I$. Subtracting \eqref{elliptic nl 2} from \eqref{elliptic sys}, we have the following coupled elliptic system for $v_j=u_j-u'_j$:
    \begin{align}\label{w sys}
	-\nabla\cdot S\nabla v_j+W_{ij}\cdot\nabla v_i=\nabla\cdot (S-S')\nabla u'_j+(W'_{ij}-W_{ij})\cdot\nabla u'_i, \quad v_j|_{\partial X}=0
    \end{align}
    The proof is now a consequence of the Fredholm alternative (as in  \cite[Theorem 2.9]{Bal2012e}). We recast \eqref{w sys} as an integral equation. Denote the operator $L_0=-\nabla\cdot(S\nabla)$ and define $L_0^{-1}: H^{-1}(X) \ni f\mapsto v\in H_0^1(X)$, where $v$ is the unique solution to the equation
    \begin{align*}
	-\nabla\cdot(S\nabla v) = f \quad (X), \quad v|_{\partial X} = 0.
    \end{align*}
    By the Lax-Milgram theorem, we have $\|v\|_{H^1_0(X)}\le C \|f\|_{H^{-1}(X)}$, where $C$ only depends on $X$ and $S$. Thus $L_0^{-1}:H^{-1} (X)\to H_0^1(X)$ is continuous, and by Rellich imbedding, $L_0^{-1}: L^2(X) \to H_0^1(X)$ is compact. Define the vector space $\H = (H_0^1 (X))^n$, $\bfv = (v_1,\dots,v_n)$, $\bfh = (L_0^{-1}f_1,\dots,L_0^{-1}f_n)$, where $f_j=\nabla\cdot (S-S')\nabla u'_j+(W'_{ij}-W_{ij})\cdot\nabla u'_i$, and the operator $\bfP:\H\to\H$ by,
    \begin{align*}
	\bfP:\H\ni \bfv\to \bfP\bfv:=(L_0^{-1}(W_{i1}\cdot\nabla v_i),\cdots,L_0^{-1}(W_{in}\cdot\nabla v_i))\in \H.
    \end{align*}
    Since the $W_{ij}$ are bounded, the differential operators $W_{ij}\cdot\nabla:H^1_0\to L^2$ are continuous. Together with the fact that $L_0^{-1}:L^2\to H^1_0$ is compact, we get that $\bfP:\H\to\H$ is compact. After applying the operator $L_0^{-1}$ to \eqref{elliptic sys}, the elliptic system is reduced to the following Fredholm equation:
    \begin{align*}
	(\bfI + \bfP)\bfv = \bfh.
    \end{align*}
    By the Fredholm alternative, if $-1$ is not an eigenvalue of $\bfP$, then $\bfI+\bfP$ is invertible and bounded $\|\bfv\|_{\H} \le \|(\bfI+\bfP)^{-1}\|_{\L(\H)} \|\bfh\|_{\H}$. Since $L_0^{-1}:H^{-1}(X)\to H^1_0(X)$ is continuous, $\bfh$ in $(H_0^1 (X))^n$ is bounded by $\bff=(f_1,\cdots,f_n)$ in $ (H^{-1} (X))^n$.
    \begin{align*}
	\|\bfh\|_{\H}\le \|L_0^{-1}\|_{\L(H^{-1},H_0^1)}\|\bff\|_{H^{-1}(X)}.
    \end{align*}
    Then we have the estimate, 
    \begin{align*}
	\|\bfv\|_{\H} \le \|(\bfI+\bfP)^{-1}\|_{\L(\H)}\|L_0^{-1}\|_{\L(H^{-1},H^1_0)} \|\bff\|_{H^{-1}(X)}
    \end{align*}
    Noting that $L_0^{-1}$ is continuous and the RHS of \eqref{w sys} is expressed by $H_I-H'_I$ and their derivatives up to second order, we have the stability estimate
    \begin{align*}
	\|\bfu-\bfu'\|_{H^1_0(X)}\leq C \|H_I-H'_I\|_{H^1(X)}
    \end{align*}
    where $C$ depends on $H_I$ but can be chosen uniform for $H_I$ and $H'_I$ sufficiently close. Then $\gamma$ is reconstructed by $\gamma = H [\nabla U]^{-1}$ and $\nabla\times\gamma^{-1}$ by \eqref{tensor coefficient}, with a stability of the form
    \begin{align*}
	\|\gamma-\gamma'\|_{L^2(X)} + \|\nabla \times(\gamma^{-1}-\gamma'^{-1}))\|_{L^2(X)} \leq C \|H_I-H'_I\|_{H^1(X)}.
    \end{align*}
\end{proof}

\section{What tensors are reconstructible ?}\label{sec:recons}

\subsection{Test cases}

\paragraph{Constant tensors.}

We first prove that Hypotheses \ref{2 sol}-\ref{hy:full rank} can be fulfilled with explicit constructions in the case of constant coefficients. 

\begin{proof}[Proof of Proposition \ref{prop:identity}]
    Hypotheses \ref{main hypo} is trivially satisfied throughout $X$ by choosing the collection of solutions $u_i(x) = x_i$ for $1\le i\le n$, then Hypothesis \ref{2 sol} is fulfilled by picking any two distinct solutions of the above family. \\
    {\bf Fulfilling Hypothesis \ref{hyp W}.} Let us pick 
    \begin{align}
	\begin{split}
	    u_i(x):&=x_i, \quad 1\le i\le n, \\
	    u_{n+1}(x):&= \frac{1}{2}x^T\gamma_0^{-\frac{1}{2}}\sum_{j=1}^n t_j(\bfe_j\otimes \bfe_j)\gamma_0^{-\frac{1}{2}}x, \quad \sum_{j=1}^n t_j=0, \quad t_p\neq t_q \quad \text{if} \quad p\neq q, \\
	    u_{n+2}(x):&= \frac{1}{2}x^T\gamma_0^{-\frac{1}{2}} \sum_{j=1}^{n-1}(\bfe_j\otimes \bfe_{j+1}+\bfe_{j+1}\otimes \bfe_{j})\gamma_0^{-\frac{1}{2}}x.	    
	\end{split}
	\label{eq:constantsol}	
    \end{align}
    In particular, $H = \gamma_0$ and $Z_i = \nabla^2 u_{n+i}$ for $i=1,2$, do not depend on $x$ and admit the expression
    \begin{align*}
	Z_1=\gamma_0^{-\frac{1}{2}}\sum_{j=1}^n t_j(\bfe_j\otimes \bfe_j)\gamma_0^{-\frac{1}{2}} \qandq 
	Z_2=\gamma_0^{-\frac{1}{2}}\sum_{j=1}^{n-1} (\bfe_j\otimes \bfe_{j+1}+\bfe_{j+1}\otimes \bfe_{j})\gamma_0^{-\frac{1}{2}}.
    \end{align*}
    We will show that the ($x$-independent) space 
    \begin{align*}
	\W = \text{span} \left\{ (Z_1 H^T \Omega)^{sym}, (Z_2 H^T \Omega)^{sym}, \Omega\in A_n(\Rm) \right\}
    \end{align*}
    has codimension one in $S_n(\Rm)$ by showing that $\W^\perp \subset \Rm \gamma_0$, the other inclusion $\supset$ being evident.
    
    Let $A\in S_n(\Rm)$ and suppose that $A\perp\W$, we aim to show that $A$ is proportional to $\gamma_0$. The symmetry of $A Z_1 H^T$ implies that $\sum_{j=1}^n t_j \bfe_j\otimes \bfe_j\gamma_0^{-\frac{1}{2}}A\gamma_0^{-\frac{1}{2}}$ is symmetric. Denote $B=\gamma_0^{-\frac{1}{2}}A\gamma_0^{-\frac{1}{2}}\in S_n(\Rm)$, we deduce that
    \begin{align*}
	t_iB_{ij}=t_jB_{ji}, \quad \text{for} \quad 1\le i,j\le n.
    \end{align*}
    Since $B$ is symmetric and $t_i\neq t_j$ if $i\neq j$, the above equation gives that $B_{ij}=0$ for $i\neq j$, thus $B$ is a diagonal matrix, i.e. $B = \sum_{i=1}^n B_{ii} \bfe_i\otimes\bfe_i$. The symmetry of  $AZ_2H^T$ implies that $\sum_{j=1}^{n-1}(\bfe_j\otimes \bfe_{j+1}+\bfe_{j+1}\otimes \bfe_{j})\gamma_0^{-\frac{1}{2}}A\gamma_0^{-\frac{1}{2}}$ is symmetric, which means that 
    \begin{align*}
	\sum_{\substack{1\le i\le n\\ 1\le j \le n-1}}B_{ii}(\bfe_j\otimes \bfe_{j+1}+\bfe_{j+1}\otimes \bfe_{j})(\bfe_{i}\otimes \bfe_{i})=\sum_{\substack{1\le i\le n\\ 1\le j \le n-1}}B_{ii}(\bfe_{i}\otimes \bfe_{i})(\bfe_j\otimes \bfe_{j+1}+\bfe_{j+1}\otimes \bfe_{j})
    \end{align*}
    Write the above equation explicitly, we get
    \begin{align*}
	\sum_{j=1}^{n-1}B_{j+1,j+1}\bfe_j\otimes \bfe_{j+1}+B_{jj}\bfe_{j+1}\otimes \bfe_{j}=\sum_{j=1}^{n-1}B_{jj}\bfe_j\otimes \bfe_{j+1}+B_{j+1,j+1}\bfe_{j+1}\otimes \bfe_{j}
    \end{align*}
    Which amounts to 
    \begin{align*}
	\sum_{j=1}^{n-1}(B_{j+1,j+1}-B_{jj})(\bfe_{j+1}\otimes \bfe_{j}-\bfe_{j+1}\otimes \bfe_{j})=0
    \end{align*}
    Notice that $\{\bfe_{j+1}\otimes \bfe_{j}-\bfe_{j+1}\otimes \bfe_{j}\}_{1\le j\le n-1}$ are linearly independent in $A_n(\Rm)$, so $B_{j+1,j+1}=B_{jj}$ for $1\le j\le n-1$, i.e. $B$ is proportional to the identity matrix. This means that $A$ must be proportional to $\gamma_0$ and thus $\W^\perp \subset \Rm \gamma_0$. Hypothesis \ref{hyp W} is fulfilled throughout $X$. \\

    {\bf Fulfilling Hypothesis \ref{hy:full rank} with $\gamma = \Imm_n$.} We split the proof according to dimension.
    \begin{description}
	\item[Even case $n = 2m$.] Suppose that $n=2m$, pick $u_i=x_i$ for $1\le i\le n$, $u_{n+1}=\sum_{i=1}^m x_{2i-1}x_{2i}$ and $u_{n+2}=\sum_{i=1}^m \frac{(x_{2i-1}^2-x_{2i}^2)}{2}$. Then simple calculations show that 
	    \begin{align*}
		Z_1=\sum_{i=1}^m (\bfe_{2i-1}\otimes \bfe_{2i}+\bfe_{2i}\otimes \bfe_{2i-1}) \qandq
		Z_2=\sum_{i=1}^m (\bfe_{2i-1}\otimes \bfe_{2i-1}-\bfe_{2i}\otimes \bfe_{2i}).
	    \end{align*}
	    We have $\det Z_1 = (-1)^m\ne 0$ so \ref{hy:full rank}.A is fulfilled. Let us choose
	    \begin{align*}
		\Omega_1 := \sum_{p=1}^m(\bfe_{2p}\otimes \bfe_{2p-1}-\bfe_{2p-1}\otimes \bfe_{2p}) \qandq \Omega_2 = 0,
	    \end{align*}
	    then direct calculations show that $S = (Z_2^{\star} Z_1^{T}\Omega_1 + HZ_1^T \Omega_2)^{sym} = \Imm_n$, which is clearly uniformly elliptic, hence \ref{hy:full rank}.B is fulfilled. 

	\item[Odd case $n=3$.] Pick $u'_i=x_i$ for $1\le i\le 3$, $u'_{3+1}=x_1x_2+x_2x_3$ and $u'_{3+2}=\frac{1}{2t_1}x^2_1+\frac{1}{2t_2}x^2_2+\frac{1}{2t_3}x^2_3$, where $t_1,t_2,t_3$ are to be chosen. In this case, $H'=\Imm_3$, $Z'_1=2(\bfe_{1}\odot \bfe_{2}+\bfe_{2}\odot \bfe_{3})$ and $(Z'_2)^{\star}=\sum_{i=1}^3t_i\bfe_{i}\otimes \bfe_{i}$ (note that $Z'_2$ fulfills \ref{hy:full rank}.A). Pick $\Omega'_1(x)=\bfe_{2}\otimes \bfe_{1}-\bfe_{1}\otimes \bfe_{2}$, $\Omega'_2(x)=\bfe_{2}\otimes \bfe_{3}-\bfe_{3}\otimes \bfe_{2}$, simply calculations show that,
	    \begin{align}\label{eq:odd matrix}
		S'=\left( (Z'_2)^{\star} (Z'_1)^{T} \Omega'_1(x)+H'(Z'_1)^{T} \Omega'_2(x) \right)^{sym} = 
		\left[\begin{array}{ccc}
			t_1 & 0 & \frac{t_3+1}{2}\\
			0 & -t_2-1 & 0 \\
			\frac{t_3+1}{2} & 0 & 1
		\end{array}\right].
	    \end{align}
	    $(t_1,t_2,t_3)$ must be such that $S'$ is positive definite and $\tr (Z'_2)=0$ (because $u'_2$ solves \eqref{eq:conductivity}). This entails the conditions 
	    \begin{align*}
		t_1>0, \quad t_1(t_2+1)<0, \quad -(t_2+1)\left( t_1-\left( \frac{t_3+1}{2} \right)^2 \right)>0 \qandq t_1=-\frac{t_2t_3}{t_2+t_3}.
	    \end{align*}
	    These conditions can be jointly satisfied for instance by picking $t_1=6$, $t_2=-2$ and $t_3=3$, thus Hypothesis \ref{hy:full rank}.B is fulfilled in the case $n=3$. 

	\item[Odd case $n=2m+3$.]  When $n=2m+3$ for $m\ge 0$, we build solutions based on the previous two cases. Let us pick 
	    \begin{align*}
		u_i &= x_i, \quad 1\le i\le n, \\
		u_{n+1} &= \sum_{i=1}^m x_{2i-1}x_{2i}+x_{2m+1}x_{2m+2}+x_{2m+2}x_{2m+3} \\
		u_{n+2} &= \sum_{i=1}^m \frac{(x_{2i-1}^2-x_{2i}^2)}{2}+\frac{1}{12}x^2_{2m+1}-\frac{1}{4}x^2_{2m+2}+\frac{1}{6}x^2_{2m+3}.	
	    \end{align*}
	    Then one can simply check that $\tilde Z_j$ is of the form
	    \begin{align*}
		\tilde Z_j=\left[ \begin{array}{c|c}
			Z_j & 0_{2m\times 3} \\ \hline  0_{3\times 2m} & Z'_j
		\end{array}\right], \quad j=1,2,
	    \end{align*}
	    where $Z_j$/$Z'_j$ are constructed as in the case $n=2m$/$n=3$, respectively. Accordingly, let us construct $\Omega_{1,2}$ by block using the previous two cases, 
	    \begin{align*}
		\tilde \Omega_j=\left[ \begin{array}{c|c}
			\Omega_j & 0_{2m\times 3} \\ \hline  0_{3\times 2m} & \Omega'_j
		\end{array}\right],
	    \end{align*}
	    and the $S$ matrix so obtained becomes 
	    \begin{align*}
		\tilde S=\left(\tilde Z_2^{\star} \tilde Z_1^{T} \tilde\Omega_1 + H\tilde Z_1^{T} \tilde\Omega_2 \right)^{sym}=\left[\begin{array}{c|c}
			\Imm_{2m} & 0_{2m\times 3} \\ \hline 0_{3\times 2m} & S'
		\end{array} \right],
	    \end{align*}
	    where $S'$ is the definite positive matrix constructed in the case $n=3$. Again, Hypothesis \ref{hy:full rank}.B is fulfilled.
    \end{description}

    {\bf Fulfilling Hypothesis \ref{hy:full rank} with $\gamma$ constant.} Let $\{v_i\}_{i=1}^{n+2}$ denote the harmonic polynomials constructed in any case above (i.e. $n$ even or odd) with $\gamma = \Imm_n$, and denote $Z_1^0, Z_2^0, H^0, \Omega_1^0, \Omega_2^0$ and $S^0 = (Z_2^{0\star}Z_1^{0T} \Omega_1^0 + H^0 Z_1^{0T} \Omega_2^0)^{sym}$ the corresponding matrices. Define here, for $1\le i\le n$, $u_i(x) := v_i(x)$ and for $i=n+1,n+2$, $u_i(x) = v_i (\gamma^{-\frac{1}{2}}x)$, all solutions of \eqref{eq:conductivity} with constant $\gamma$. Then we have that $Z_i = \gamma^{-\frac{1}{2}} Z_i^0 \gamma^{-\frac{1}{2}}$ for $i=1,2$ and $H = \gamma$. Upon defining $\Omega_i:= \gamma^{\frac{1}{2}}\Omega_i^0 \gamma^{\frac{1}{2}} \in A_n(\Rm)$ for $i=1,2$, direct calculations show that 
    \begin{align*}
	S = (Z_2^{\star}Z_1^{T} \Omega_1 + H Z_1^{T} \Omega_2)^{sym} = \gamma^{\frac{1}{2}} S^0 \gamma^{\frac{1}{2}}.
    \end{align*}
    Whenever $Z_1^0$ is non-singular, so is $Z_1$ and whenever $S_0$ is symmetric definite positive, so is $S$. The proof is complete. 
\end{proof}

\paragraph{Isotropic tensors.} As a second test case, we show that, based on the construction of complex geometrical optics (CGO) solutions, Hypothesis \ref{2 sol} can be satisfied globally for an isotropic tensor $\gamma = \beta\Imm_n$ when $\beta$ is smooth enough. CGO solutions find many applications in inverse conductivity/diffusion problems, and more recently in problems with internal functionals \cite{Bal2011a,Monard2011a,Bal2012}. As established in \cite{Bal2010}, when $\beta\in H^{\frac{n}{2}+3+\varepsilon}(X)$, one is able to construct a complex-valued solution of \eqref{eq:conductivity} of the form 
    \begin{align}
	u_\bmrho = \frac{1}{\sqrt{\beta}} e^{\bmrho\cdot x} (1+ \psi_\bmrho),
	\label{eq:urho}
    \end{align}
    where $\bmrho\in \Cm^n$ is a complex frequency satisfying $\bmrho\cdot\bmrho = 0$, which is equivalent to taking $\bmrho = \rho (\bk + i\bk^\perp)$ for some unit orthogonal vectors $\bk, \bk^\perp$ and $\rho = |\bmrho|/\sqrt{2}>0$. The remainder $\psi_{\bmrho}$ satisfies an estimate of the form $\rho\psi_{\bmrho} = \O(1)$ in $\C^1(\overline{X})$. The real and imaginary parts of $\nabla u_\bmrho$ are almost orthogonal, modulo an error term that is small (uniformly over $X$) when $\rho$ is large. We use this property here to fulfill Hypothesis \ref{2 sol}. 

\begin{proof}[Proof of Proposition \ref{prop:isotropic}] 
    Pick two unit orthogonal vectors $\bk$ and $\bk^\perp$, and consider the CGO solution $u_{\bmrho}$ as in \eqref{eq:urho} with $\bmrho = \rho(\bk + i\bk^\perp)$ for some $\rho>0$ which will be chosen large enough later. Computing the gradient of $u_{\bmrho}$, we arrive at 
    \begin{align*}
	\nabla u_{\bmrho} = e^{\bmrho\cdot x} (\bmrho + \bmphi_{\bmrho}), \quad\text{with} \quad \bmphi_\bmrho := \nabla\psi_\bmrho - \psi_{\bmrho} \nabla \log\sqrt{\beta},
    \end{align*}
    with $\sup_{\overline{X}} |\bmphi_\bmrho| \le C$ independent of $\bmrho$. Splitting into real and imaginary parts, each of which is a real-valued solution of \eqref{eq:conductivity}, we obtain the expression
    \begin{align*}
	\nabla u_\bmrho^\Re &= \frac{\rho e^{\rho\bk\cdot x}}{\sqrt{\beta}} \left( (\bk + \rho^{-1} \bmphi_\bmrho^\Re) \cos(\rho \bk^\perp\cdot x) - (\bk^\perp + \rho^{-1} \bmphi_\bmrho^\Im) \sin (\rho\bk^\perp\cdot x) \right), \\
	\nabla u_\bmrho^\Im &= \frac{\rho e^{\rho\bk\cdot x}}{\sqrt{\beta}} \left(  (\bk^\perp + \rho^{-1} \bmphi_\bmrho^\Im) \cos(\rho \bk^\perp\cdot x) + (\bk + \rho^{-1} \bmphi_\bmrho^\Re) \sin (\rho\bk^\perp\cdot x)  \right),
    \end{align*}
    from which we compute directly that 
    \begin{align*}
	|\nabla u_{\bmrho}^\Re|^2 |\nabla u_{\bmrho}^\Im|^2-(\nabla u_{\bmrho}^\Re\cdot \nabla u_{\bmrho}^\Im)^2 = \frac{\rho^2 e^{2\rho\bk\cdot x}}{\beta} ( 1 + o (\rho^{-1}) ).
    \end{align*}
    Therefore, for $\rho$ large enough, the quantity in the left-hand side above remains bounded away from zero throughout $X$, and the proof is complete.   
\end{proof}

\subsection{Push-forward by diffeomorphism}

Let $\Psi:X\to \Psi(X)$ be a $W^{1,2}$-diffeomorphism where $X$ has smooth boundary. Then for $\gamma\in\Sigma(X)$, the push-forwarded tensor $\Psi_\star \gamma$ defined in \eqref{eq:pushfwd} belongs to $\Sigma(\Psi(X))$ and $\Psi$ pushes forward a solution $u$ of \eqref{eq:conductivity} to a function $v = u\circ\Psi^{-1}$ satisfying the conductivity equation
\begin{align*}
    -\nabla_y \cdot (\Psi_\star \gamma \nabla_y v) = 0 \quad (\Psi(X)), \quad v|_{\partial (\Psi(X))} = g\circ\Psi^{-1},
\end{align*}
moreover $\Psi$ and $\Psi|_{\partial X}$ induce respective isomorphisms of $H^1(X)$ and $H^{\frac{1}{2}}(\partial X)$ onto $H^1(\Psi(X))$ and $H^{\frac{1}{2}}(\partial (\Psi(X)))$. 

\begin{proof}[Proof of Proposition \ref{prop:pushfwd}] 
    The hypotheses of interest all fomulate the linear independence of some functionals in some sense. We must see first how these functionals are push-forwarded via the diffeomorphism $\Psi$. For $1\le i\le m$, we denote $v_i := \Psi_\star u_i = u_i \circ\Psi^{-1}$ as well as $\Psi_\star H_i := [\Psi_\star\gamma] \nabla_y v_i$ where $y$ denotes the variable in $\Psi(X)$. Direct use of the chain rule allows to establish the following properties, true for any $x\in X$:
    \begin{align}
	\begin{split}
	    \nabla u_i (x) &= [D\Psi]^T(x) \nabla_y v_i (\Psi(x)), \\
	    H_i(x) &= \gamma\nabla u_i (x) = |J_\Psi|(x) [D\Psi]^{-1} \Psi_\star H (\Psi(x)), \\
	    Z_i(x) &= [D\Psi]^T (x) \Psi_\star Z_i (\Psi(x)), 	    
	\end{split}
	\label{eq:rels}
    \end{align}
    where we have defined $\Psi_\star Z_i$ the matrix with columns
    \begin{align*}
	[\Psi_\star Z_i]_{,j} = - \nabla_y  \frac{\det (\nabla_y v_1, \dots, \overbrace{\nabla_y v_{n+i}}^{j}, \dots, \nabla_y v_n)}{\det (\nabla_y v_1,\dots,\nabla_y v_n)}, \quad 1\le j\le n. 
    \end{align*}

    {\bf Hypotheses \ref{2 sol} and \ref{main hypo}.} Since $[D\Psi]$ is never singular over $X$, relations \eqref{eq:rels} show that for any $1\le k\le n$, the vectors fields $(\nabla u_1,\dots, \nabla u_k)$ are linearly dependent at $x$ if and only if the vectors fields $(\nabla_y v_1, \dots, \nabla_y v_k)$ are linearly dependent at $\Psi(x)$. The case $k=2$ takes care of Hyp. \ref{2 sol} while the case $k=n$ takes care of Hyp. \ref{main hypo}. 

    {\bf Hypothesis \ref{hyp W}.} If we denote 
    \begin{align*}
	\Psi_\star \W (\Psi(x)) = \text{span} \left\{ (\Psi_\star Z_k (\Psi_\star H)^T \Omega)^{sym}, \quad \Omega\in A_n(\Rm), 1\le k\le m \right\},
    \end{align*}
    direct computations show that 
    \begin{align*}
	\W(x) = [D\Psi(x) ]^T\cdot  \Psi_\star \W (\Psi(x)) \cdot  [D\Psi(x)],
    \end{align*}
    thus since $D\Psi(x)$ is non-singular, we have that $\dim \W(x) = \dim \Psi_\star \W (\Psi(x))$, so the statement of Proposition holds for Hyp. \ref{hyp W}. 
    
    {\bf Hypothesis \ref{hy:full rank}.} The transformation rules \eqref{eq:rels} show that $Z_1$ is nonsingular at $x$ iff $\Psi_\star Z_1$ is nonsingular at $\Psi(x)$, so the statement of the proposition holds for Hyp. \ref{hy:full rank}.A.

    Second, for two $A_n(\Rm)$-valued functions $\Omega_1(x)$ and $\Omega_2(x)$, and upon defining $\Psi_\star \Omega_1$, $\Psi_\star \Omega_2$ as in \eqref{eq:psiomega}, as well as 
    \begin{align*}
	\Psi_\star S := \left( [\Psi_\star Z_2]^{-T} [\Psi_\star Z_1]^T \Psi_\star \Omega_1 + [\Psi_\star H][\Psi_\star Z_1]^T \Psi_\star \Omega_2 \right)^{sym},
    \end{align*}
    direct use of relations \eqref{eq:rels} yield the relation
    \begin{align*}
	S(x) = [D\Psi(x)]^{-1} \cdot \Psi_\star S (\Psi(x)) \cdot [D\Psi(x)]^{-T}, \quad x\in X,
    \end{align*}
    and since $D\Psi$ is uniformly non-singular, $S$ is uniformly elliptic if and only if $\Psi_\star S$ is, so the statement of the proposition holds for Hyp. \ref{hy:full rank}.B.
\end{proof}

\subsection{Generic reconstructibility}

We now show that, in principle, any $\C^{1,\alpha}$-smooth conductivity tensor is locally reconstructible from current densities. The proof relies on the Runge approximation for elliptic equations, which is equivalent to the unique continuation principle, valid for conductivity tensors with Lipschitz-continuous components. 

This scheme of proof was recently used in the context of other inverse problems with internal functionals \cite{Bal2012,Monard2012a}, and the interested reader is invited to find more detailed proofs there. 

\begin{proof}[Proof of Proposition \ref{prop:Runge}.]
    Let $x_0\in X$ and denote $\gamma_0 := \gamma(x_0)$. We first construct solutions of the constant-coefficient problem by picking the functions defined in \eqref{eq:constantsol} (call them $v_1,\dots, v_{n+2}$) and by defining, for $1\le i\le n+2$, $u_i^0(x) := v_i(x) - v_i(x_0)$. These solutions satisfy $\nabla\cdot(\gamma_0\nabla u_i)=0$ everywhere and fulfill Hypotheses \ref{main hypo} and \ref{hyp W} globally. 

    Second, from solutions $\{u_i^0\}_{i=1}^{n+2}$, we construct a second family of solutions $\{u_i^r\}_{i=1}^{n+2}$ via the following equation
    \begin{align}
	\nabla\cdot(\gamma \nabla u_i^r) = 0\quad (B_{3r}), \quad u_i^r|_{\partial B_{3r}} = u_i^0, \qquad 1\le i\le n+2,
	\label{eq:localsol}
    \end{align}
    where $B_{3r}$ is the ball centered at $x_0$ and of radius $3r$, $r$ being tuned at the end. The maximum principle as well as interior regularity results for elliptic equations allow to deduce the fact that 
    \begin{align}
	\lim_{r\to 0} \quad  \max_{1\le i\le n+2} \|u_i^r - u_i^0\|_{\C^2(B_{3r})} = 0.
	\label{eq:localestimate}
    \end{align}
    
    Third, assuming that $r$ has been fixed at this stage, the Runge approximation property allows to claim that for every $\varepsilon >0$ and $1\le i\le n+2$, there exists $g_i^\varepsilon\in H^{\frac{1}{2}} (\partial X)$ such that 
    \begin{align}
	\|u_i^\varepsilon - u_i^r\|_{L^2(B_{3r})} \le \varepsilon, \where\ u_i^\varepsilon \text{ solves } \eqref{eq:conductivity} \text{ with } u_i^\varepsilon|_{\partial X} = g_i^\varepsilon,
	\label{eq:L2est}
    \end{align}
    which, combined with interior elliptic estimates, yields the estimate 
    \begin{align*}
	\|u_i^\varepsilon - u_i^r\|_{\C^2 (\overline{B_r})} \le \frac{C}{r^2} \|u_i^\varepsilon - u_i^r\|_{L^\infty(B_{2r})} \le \frac{C}{r^2} \varepsilon,
    \end{align*}
    Since $r$ is fixed at this stage, we deduce that 
    \begin{align}
	\lim_{\varepsilon\to 0} \quad \max_{1\le i\le n+2} \|u_i^\varepsilon-u_i^r\|_{\C^2(B_r)} = 0.
	\label{eq:epsilonestimate}
    \end{align}
    
    Completing the argument, we recall that Hypotheses \ref{main hypo} and \ref{hyp W} are characterized by continuous functionals (say $f_2$ and $f_3$) in the topology of $\C^{2,\alpha}$ boundary conditions. While the first step established that $f_2 >0$ and $f_3>0$ for the constant-coefficient solutions, limits \eqref{eq:localestimate} and \eqref{eq:epsilonestimate} tell us that there exists a small $r>0$, then a small $\varepsilon>0$ such that $\max_{1\le i\le n+2} \|u_i^\varepsilon-u_i^0\|_{\C^2(B_r(x_0))}$ is so small that, by the continuity of $f_2$ and $f_3$, these functionals remain positive. Hypotheses \ref{main hypo} and \ref{hyp W} are thus satisfied over $B_r$ by the family $\{u_i^\varepsilon\}_{i=1}^{n+2}$ which is controlled by boundary conditions. The proof is complete.
\end{proof}

\begin{remark}[On generic global reconstructibility] Let us mention that from the local reconstructibility statement above, one can establish a global reconstructibility one. Heuristically, by compactness of $\overline{X}$, one can cover the domain with a finite number of either neighborhoods as above or subdomains diffeomorphic to a half-ball if the point $x_0$ is close to $\partial X$, over each of which $\gamma$ is reconstructible. One can then patch together the local reconstructions using for instance a partition of unity, and obtain a globally reconstructed $\gamma$. The additional technicalities that this proof incurs may be found in \cite{Bal2012}.

    As a conclusion, for any $\C^{1,\alpha}$-smooth tensor $\gamma$, there exists a finite $N$ and non-empty open set $\O\subset (\C^{2,\alpha}(\partial X))^N$ such that any $\{g_i\}_{i=1}^N\in \O$ generates current densities that reconstruct $\gamma$ uniquely and stably (in the sense of estimate \eqref{eq:jointstab2}) throughout $X$. 
\end{remark}

\appendix

\section{Exterior calculus and notations}\label{app}

Throughout this paper, we use the following convention regarding exterior calculus. Because we are in the Euclidean setting, we will avoid the flat operator notation by identifying vector fields with one-forms via the identification $\bfe_i\equiv \bfe^i$ where $\{\bfe_i\}_{i=1}^n$ and $\{\bfe^i\}_{i=1}^n$ denote bases of $\Rm^n$ and its dual, respectively. In this setting, if $V = V^i \bfe_i$ is a vector field, $d V$ denotes the two-vector field 
\begin{align*}
    dV = \sum_{1\le i<j \le n} (\partial_i V^j - \partial_j V^i) \bfe_i\wedge\bfe_j.
\end{align*}
A two-vector field can be paired with two other vector fields via the formula 
\begin{align*}
    A \wedge B (C,D) = (A\cdot C)(B\cdot D) - (A\cdot D) (B\cdot C), 
\end{align*}
which allows to make sense of expressions of the form 
\begin{align*}
    dV (A, \cdot) = \sum_{1\le i<j \le n} (\partial_i V^j - \partial_j V^i) ((A\cdot\bfe_i) \bfe_j - (A\cdot\bfe_j) \bfe_i).
\end{align*}
Note also the following well-known identities for $f$ a smooth function and $V$ a smooth vector field, rewritten with the notation above:
\begin{align}
    d (\nabla f) &= 0, \qquad f\in \C^2(X),     \label{eq:d2zero} \\
    d (fV) &= \nabla f \wedge V + fdV. \label{eq:dfV}
\end{align}

\section*{Acknowledgments}

This work was supported in part by NSF grant DMS-1108608 and AFOSR Grant NSSEFFFA9550-10-1-0194. FM is partially supported by NSF grant DMS-1025372.

\end{document}